\documentclass[11pt]{article}
\usepackage{euscript}
\usepackage{pstricks}
\usepackage{theorem}
\usepackage{amsmath}
\usepackage{amsfonts}
\usepackage{amssymb}
\usepackage{latexsym}
\usepackage[dvips]{graphicx}
\usepackage{rotating}
\newtheorem{theorem}{\bf Theorem}[section]
\newtheorem{lemma}[theorem]{\bf Lemma}

\newtheorem{defn}[theorem]{\bf Definition}
\newtheorem{remark}[theorem]{\bf Remark}
\newenvironment{proof}{\noindent{\em Proof:}}{\quad \hfill$\Box$\vspace{2ex}}

\def \bN {\Bbb N}
\def \bZ {\Bbb Z}
\def \bE {\Bbb E}

\def \bR {\Bbb R}

\def \bH {\Bbb H}

\def \bB {\Bbb B}
\def \bC {\Bbb C}
\def \bA {\Bbb A}

\def \bT {\Bbb T}

\def \bY {\Bbb Y}
\def \bV {\Bbb V}

\def \and {\, \mbox{\rm and}\, }

\def \span {\,{\rm span}\,}

\def \Re {\,{\rm Re}\,}
\def \Im {\,{\rm Im}\,}

\makeatletter

\newcommand{\Rmnum}[1]{\expandafter\@slowromancap\romannumeral #1@}
\makeatother
\begin{document}
\title{\sffamily Multivariate approximation by translates of the Korobov function on Smolyak grids
}
\author{Dinh D\~ung$^a$\footnote{Corresponding author. Email: dinhzung@gmail.com.},
Charles A. Micchelli$^b$ \\\\
$^a$ Vietnam National University, Hanoi, Information Technology Institute \\
144 Xuan Thuy, Hanoi, Vietnam  \\\\
$^b$Department of Mathematics and Statistics, SUNY Albany \\
Albany, 12222, USA \\\\
}
\date{\ttfamily April 16, 2013 --  Version R1}
 \tolerance 2500
\def\II{{\mathbb I}}
\def\TT{{\mathbb T}}
\def\CC{{\mathbb C}}
\def\ZZ{{\mathbb Z}}
\def\NN{{\mathbb N}}
\def\RR{{\mathbb R}}
\def\EEm{{\mathbb E}^m}
\def\Ppnu{{\mathcal P}_\nu}
\def\Hh{{\mathcal H}}
\def\Tt{{\mathcal T}}
\def\NNd{{\mathbb N}^d}
\def\RRd{{\mathbb R}^d}
\def\RRdp{{\mathbb R}^d_+}
\def\ZZd{{\mathbb Z}^d}
\def\ZZdp{{\mathbb Z}^d_+}
\def\TTd{{\mathbb T}^d}
\def\kr{K^r_p}
\def\proof{\noindent{\it Proof}. \ignorespaces}
\def\endproof{\vbox{\hrule height0.6pt\hbox{\vrule height1.3ex%
width0.6pt\hskip0.8ex\vrule width0.6pt}\hrule height0.6pt}}
\newlength{\fixboxwidth}
\setlength{\fixboxwidth}{\marginparwidth}
\addtolength{\fixboxwidth}{-0pt}
\newcommand{\fix}[1]{\marginpar{\fbox{\parbox{\fixboxwidth}
{\raggedright\tiny #1}}}}

\maketitle

\begin{abstract}
For a set $\mathbb{W} \subset L_p(\TTd)$, $1 < p < \infty$, of  multivariate periodic
functions on the torus $\TTd$ and a given function $\varphi \in
L_p(\TTd)$, we study the approximation in the $L_p(\TTd)$-norm of
functions $f \in \mathbb{W}$ by arbitrary linear combinations of $n$
translates of $\varphi$.  For $\mathbb{W} = U^r_p(\bT^d)$ and $\varphi = \kappa_{r,d}$, we prove upper bounds of the worst case error of this approximation where  $U^r_p(\bT^d)$ is the unit ball in the Korobov space $K^r_p(\bT^d)$
and $\kappa_{r,d}$ is the associated Korobov function. To obtain the upper bounds, we construct
approximation methods based on sparse Smolyak grids. The case $p=2, \ r
> 1/2$, is especially important since $K^r_2(\bT^d)$ is a reproducing
kernel Hilbert space, whose reproducing kernel is a translation
kernel determined by $\kappa_{r,d}$. We also
provide lower bounds of the optimal approximation on the best choice of $\varphi$.


\end{abstract}

\section{Introduction}
The $d$-dimensional torus denoted by $\bT^d$ is the cross product of
$d$ copies of the interval $[0,2\pi]$ with the identification of the end points. When $d=1$, we merely
denote the $d$-torus by $\bT$. Functions on $\bT^d$
are identified with functions on $\RRd$ which are $2\pi$ periodic in
each variable. We shall denote by $L_p(\TTd), \ 1 \le p < \infty$,
the space of integrable functions on $\TTd$ equipped with the norm
\begin{equation}\label{p_norm}
\|f\|_p \ := (2\pi)^{-d/p}\left(\int_{\TTd} |f(\bold{x})|^p
d\bold{x}\right)^{1/p}.
\end{equation}
We will consider only real valued functions on $\TTd$. However, all the results in this paper are true for the complex setting. Also, we will use the Fourier series of a real valued function in complex form and somewhere estimate its $L_p(\TTd)$-norm  via the $L_p(\TTd)$-norm  of its complex valued components which is defined as in 
\eqref{p_norm}.  

For vectors ${\bold x}:=(x_l:l\in N[d])$ and
$\bold{y}:=(y_l:l\in N[d])$ in $\bT^d$ we use
$(\bold{x},\bold{y}):=\sum_{l\in N[d]}x_ly_l$ for the inner product
of $\bold{x}$ with $\bold{y}$. Here, we use the notation $N[m]$ for
the set $\{1,2,\ldots,m\}$ and later we will use $Z[m]$ for the set
$\{0,1,\ldots,m-1\}$. Also, for notational
convenience we allow $N[0]$ and $Z[0]$ to stand for the empty set. Given any integrable function
$f$ on $\TTd$ and any lattice vector $\bold{j}=(j_l: l\in N[d]) \in
\ZZd$, we let ${\hat f}(\bold{j})$ denote the $\bold{j}$-th Fourier
coefficient of $f$ defined by
\[
{\hat f}(\bold{j}) \ := \ (2\pi)^{-d}\int_{\TTd} f(\bold{x}) \,
\chi_{- \bold j}({\bold x}) \, d\bold{x},
\]
where we define the exponential function $\chi_{\bold j}$ at ${\bold
x}\in\TTd$ to be $\chi_{\bold j}({\bold x})=e^{i({\bold j},{\bold
x})}.$ Frequently, we use the superscript notation $\bB^d$ to denote
the cross product of a given set $\bB$. 

The convolution of two functions $f_1$ and $f_2$ on
$\bT^d$, denoted by $f_1*f_2$, is defined at $\bold{x}\in\bT^d$ by
equation
\[
(f_1*f_2)(\bold{x}) \ := (2\pi)^{-d}\int_{\TTd} f_1(\bold{x}) \,
f_2(\bold{x} - \bold{y}) \, d\bold{y},
\]
whenever the integrand is in $L_1(\bT^d)$. We are interested in
approximations of functions from the Korobov space $K^r_p(\bT^d)$ by
arbitrary linear combinations of $n$ arbitrary shifts of the Korobov
function $\kappa_{r,d}$ defined below. The case $p=2$ and $r>1/2$ is
especially important, since $K^r_2(\bT^d)$ is a reproducing kernel
Hilbert space.

In order to formulate the setting for our problem, we establish some
necessary definitions and notation. For a given $r > 0$ and a
lattice vector $\bold{j}:=(j_l:l\in N[d])\in\bZ^d$ we define the
scalar $\lambda_{\bold j}$ by the equation
\[ \lambda_{\bold j}
 := \
\prod_{l\in N[d]} \, \lambda_{j_l},
\]
where
\begin{equation*}
\lambda_{l}
 := \
\begin{cases}
|l|^r &, \ l\in\ZZ\setminus\{0\},\\
1 &, \quad \mbox{otherwise}.
\end{cases}
\end{equation*}

\begin{defn}\label{kappa_rd}
The Korobov function $\kappa_{r,d}$ is defined at $\bold{x}\in\bT^d$ by
the equation
\[
\kappa_{r,d}(\bold{x}) \ := \ \sum_{\bold{j} \in \ZZd}
\lambda_{\bold j}^{-1} \, \chi_{\bold j}({\bold x})
\]
and the corresponding Korobov space is
\[K^r_p(\bT^d):=\{f: f =
\kappa_{r,d}*g, \ g \in L_p(\TTd)\}
\]
 with norm
\[
\|f\|_{K^r_p(\TT^d)} \ := \|g\|_p.
\]
\end{defn}
\begin{remark}
The univariate Korobov function $\kappa_{r,1}$ shall always be
denoted simply by $\kappa_r$ and therefore $\kappa_{r,d}$ has at
${\bold x}=(x_l:l\in N[d])$ the alternate tensor product
representation
\begin{equation}\nonumber
\kappa_{r,d}({\bold x})=\prod_{l\in N[d]}\kappa_r(x_l)
\end{equation}
because, when ${\bold j}=(j_l:l\in N[d])$ we have that
\[
\kappa_{r,d}({\bold x})
=\sum_{{\bold j}\in\bZ^d}\lambda_{\bold j}^{-1}\chi_{\bold j}({\bold x})
=\sum_{{\bold j} \in \bZ^d}\prod_{l\in N[d]}(\lambda_{j_l}^{-1}\chi_{j_l}(x_l))
=\prod_{l\in N[d]}\cdot\sum_{j\in\bZ}\lambda_j^{-1}\chi_j(x_l).
\]
\end{remark}

\begin{remark}
For $1 \le p \le \infty$ and $r > 1/p$, we have the embedding ${K^r_p(\TT^d)} \hookrightarrow C(\bT^d)$, i.e., we can consider ${K^r_p(\TT^d)}$ as a subset of  $C(\bT^d)$. Indeed, for $d=1$, it follow from the embeddings  
\[
{K^r_p(\TT)} \hookrightarrow B^r_{p,\infty}(\bT)\hookrightarrow B^{r-1/p}_{\infty,\infty}(\bT)\hookrightarrow C(\bT),
\]
 where $B^r_{p,\infty}(\bT)$ is the Nikol'skii-Besov space. See the proof of the embedding ${K^r_p(\TT)} \hookrightarrow B^r_{p,\infty}(\bT)$ in \cite[Theorem I.3.1, Corollary 2 of Theorem I.3.4, (I.3.19)]{Te93}. Corresponding relations for ${K^r_p(\TT^d)}$ can be found in \cite[III.3]{Te93}.
\end{remark}

\begin{remark}
Since $\hat{\kappa}_{r,d}({\bf j})\neq 0$ for any ${\bf j}\in\ZZd$
it readily follows that $\|\cdot\|_{K^r_p(\TT^d)}$ is a norm.
Moreover, we point out that the univariate Korobov function is
related to the one-periodic extension of Bernoulli polynomials.
Specifically, if we denote the one-periodic extension of the
Bernoulli polynomial as $\bar{B}_n$ then for $t\in\bT$, we have that
\[
\bar{B}_{2m}(t)=\frac{2m!}{(2\pi i)^{2m}}(1-\kappa_{2m}(2\pi t)).
\]
\end{remark}

When $p=2$ and $r>1/2$ the kernel $K$ defined at $\bold{x}$ and
$\bold{y}$ in $\bT^d$ as $K(\bold{x,y}):=\kappa_{2r,d}(\bold{x-y})$
is the reproducing kernel for the Hilbert space $K^r_2(\bT^d)$. This
means, for every function $f\in K^r_2(\bT^d)$ and
$\bold{x}\in\bT^d$, we have that
\begin{equation} \nonumber
f(\bold{x}) \ = (f,K(\cdot,\bold{x}))_{K^r_2(\bT^d)},
\end{equation}
where $(\cdot,\cdot)_{K^r_2(\bT^d)}$ denotes the inner product on
the Hilbert space $K^r_2(\bT^d)$. For a definitive treatment of
reproducing kernel, see, for example, \cite{Ar50}.

Korobov spaces $K^r_p(\bT^d)$ are important for the study of smooth
multivariate periodic functions. They are sometimes called periodic
Sobolev spaces of dominating mixed smoothness and are useful for the
study of multivariate approximation and integration, see, for
example, the books \cite{Te93} and \cite{NW08}.

The linear span of the set of functions $\{\kappa_{r,d}(\cdot -
\bold{y}): \bold{y} \in \TTd\}$ is dense in the Hilbert space
$K^r_2(\bT^d)$. In the language of Machine Learning, this means that
the reproducing kernel for the Hilbert space is universal. The
concept of universal reproducing kernel has significant statistical
consequences in Machine Learning. In the paper \cite{MXZ06}, a
complete characterization of universal kernels is given in terms of
its feature space representation. However, no information is
provided about the degree of approximation. This unresolved question
is the main motivation of this paper and we begin to address it in
the context of the Korobov space $K^r_2(\bT^d)$. Specifically, we
study approximations in the $L_2(\bT^d)$ norm of functions in
$K^r_2(\bT^d)$ when $r> 1/2$ by linear combinations of $n$
translates of the reproducing kernel, namely, $\kappa_{r,d}(\cdot -
{\bold y}_l), \ {\bold y}_l \in \TTd, l\in N[n]$. We shall also
study  this problem in the space $L_p(\TTd)$, $1 < p < \infty$ for
$r > 1$, because the linear span of the set of functions
$\{\kappa_{r,d}(\cdot - \bold{y}): \bold{y} \in \TTd\}$, is also
dense in the Korobov space $K^r_p(\bT^d)$.

For our purpose in this paper, the following concept is essential.
Let $\mathbb{W} \subset L_p(\TTd)$ and $\varphi \in L_p(\TTd)$ be a
given function. We are interested in the approximation in
$L_p(\TTd)$-norm of all functions $f \in \mathbb{W}$ by arbitrary
linear combinations of $n$ translates of the function $\varphi$,
that is, the functions $\varphi(\cdot - {\bold y}_l), \ {\bold y}_l
\in \TTd$ and measure the error in terms of the quantity
\begin{equation}  \nonumber
M_n(\mathbb{W},\varphi)_p
 := \
\sup\{ \ \inf\{ \ \|f - \sum_{l\in N[n]} c_l \varphi(\cdot -
{\bold y}_l)\|_p:c_l\in\mathbb{R},{\bold y}_l\in\TTd\}:{f \in
\mathbb{W}}\}.
\end{equation}
The aim of the present paper is to investigate the convergence rate,
when $n\rightarrow\infty$, of  $M_n(U^r_p(\bT^d), \kappa_{r,d})_p$
where $U^r_p(\bT^d)$ is the unit ball in $K^r_p(\bT^d)$. We shall
also obtain a lower bound for the convergence rate as
$n\rightarrow\infty$ of the quantity
\begin{equation}  \nonumber
M_n(U^r_2(\bT^d))_2  := \
\inf\{ M_n(U^r_2(\bT^d),\varphi)_2:\varphi \in L_2(\TTd)\}
\end{equation}
which
gives information about the best choice of $\varphi$.

The paper \cite{Ma05} is directly related to the questions we
address in this paper, and we rely upon some results from
\cite{Ma05} to obtain lower bound for the quantity of
$M_n(U^r_p(\bT^d))_p$. Related material can be found in the papers
\cite{Ma03} and \cite{MM01}. Here, we shall provide upper bounds for
$M_n(U^r_p(\bT^d), \kappa_{r,d})_p$ for $1 < p < \infty, r>1,
p\neq2$ and $r>1/2$ for $p=2$, as well as lower bounds for
$M_n(U^r_2(\bT^d))_2$. To obtain our upper bound, we construct
approximation methods based on sparse Smolyak grids. Although these
grids have a significantly smaller number of points than the
corresponding tensor product grids, the error approximation remains
the same.  Smolyak grids \cite{S63} and the related notion of hyperbolic cross introduced by Babenko \cite{Ba60},
are useful for high dimensional approximation problems, see,
for example, \cite{DU13} and \cite{GN09}. For recent results on
approximations and sampling on Smolyak grids see, for example, \cite{BG04},
\cite{DD11}, \cite{SU07}, and \cite{SU11}.

To describe the main results of our paper, we recall the following
notation. Given two sequences $\{a_l:l\in \NN\}$ and $\{b_l:l\in
\NN\}$, we write $a_l\ll b_l$ provided there is a positive constant
$c$ such that for all $l\in \NN$, we have that $a_l \le cb_l$.  When
we say that $a_l \asymp b_l$ we mean that both $a_l \ll b_l$ and
$b_l \ll a_l$ hold. The main theorem of this paper is the following
fact.
\begin{theorem}\label{Maintheorem}
If $1 < p < \infty, \ p \ne 2, \ r > 1$ or $p = 2, \ r > 1/2$, then
\begin{equation} \label{[mainresult1]}
M_n(U^r_p(\bT^d), \kappa_{r,d})_p \ \ll \ n^{-r} (\log n)^{r(d-1)},
\end{equation}
while for $r > 1/2$, we have that
\begin{equation} \label{[mainresult2]}
n^{-r} (\log n)^{r(d-2)} \ \ll \ M_n(U^r_2(\bT^d))_2 \ \ll \ n^{-r}
(\log n)^{r(d-1)}.
\end{equation}
\end{theorem}

This paper is organized in the following manner. In Section
\ref{Univariate Approximation}, we give the necessary background
from Fourier analysis, construct methods for approximation of
functions from the univariate Korobov space $K_p^r(\bT)$ by linear
combinations of translates of the Korobov function $\kappa_{r}$ and
prove an upper bound for the approximation error. In Section
\ref{Multivariate approximations}, we extend the method of
approximation developed in Section \ref{Univariate Approximation} to
the multivariate case and provide an upper bound for the
approximation error. Finally, in Section \ref{Convergence rate and
optimality}, we provide the proof of the Theorem \ref{Maintheorem}.

\section{Univariate Approximation} \label{Univariate Approximation}
\setcounter{equation}{0}

We begin this section by introducing the $m$-th Dirichlet function,
denoted by $D_m$, and defined at $t\in\TT$ as
\begin{equation} \nonumber
D_m(t):= \ \sum_{|l|\in Z[m+1]}  \chi_l(t) \ = \ \frac{\sin((m +
1/2)t)}{\sin(t/2)}
\end{equation}
and corresponding $m$-th Fourier projection of $f\in L_p(\bT)$,
denoted by $S_m(f)$, and given as $S_m(f):=D_m*f$. The following
lemma is a basic result.
\begin{lemma} \label{Ineq:[|f - S_n(f)|_p]}
If $1 < p < \infty$ and $r > 0$, then there exists a positive constant $c$ such
that for any $m\in\bN$, $f \in K^r_p(\bT)$ and $g\in L_p(\bT^d)$ we
have 
\begin{equation}\label{f-S_m(f)}
\|f - S_m(f)\|_p \ \le \ c \, m^{-r} \|f\|_{K^r_p(\bT)}
\end{equation}
and
\begin{equation}\label{S_m(g)}
\|S_m(g)\|_p\leq c \,\|g\|_p.
\end{equation}
\end{lemma}
\begin{remark}
The proof of inequality (\ref{f-S_m(f)}) is easily verified while
inequality (\ref{S_m(g)}) is given in Theorem 1, page 137, of
\cite{Bari}.
\end{remark}
The main purpose of this section is to introduce a linear operator,
denoted as $Q_m$, which is constructed from the $m$-th Fourier
projection and prescribed translate of the Korobov function
$\kappa_r$, needed for the proof of Theorem \ref{Maintheorem}.
Specifically, for $f\in K_p^r(\bT)$ we define $Q_m(f)$, where $f$ is
represented as $f=\kappa_r*g$ for $g \in L_p(\TT),$ to be

\begin{equation}\label{Q_mf}Q_m(f)\ := \ (2m + 1)^{-1}\sum_{l\in Z[2m+1]}
S_m(g)\left(\frac{2\pi l}{2m + 1}\right) \kappa_r\left(\cdot -
\frac{2\pi l}{2m + 1}\right).
\end{equation}

Our main observation in this section is to establish that the
operator $Q_m$ enjoys the same error bound which is valid for $S_m$.
We state this fact in the theorem below.
\begin{theorem} \label{Theorem:[M_n(f)_p<]}
If $1 < p < \infty$ and $\ r > 1$, then there is a positive constant
$c$ such that for all $m \in \mathbb{N}$ and $f \in K^r_p(\bT)$, we
have that
\begin{equation} \nonumber
\|f - Q_m(f)\|_p \ \le \ c \,  m^{-r}\|f\|_{K^r_p(\bT)}
\end{equation}
and
\begin{equation}\label{Q_m(g)_p}
\|Q_m(f)\|_p\leq c \, \|f\|_{K^r_p(\bT)}.
\end{equation}
\end{theorem}
The idea in the proof of Theorem \ref{Theorem:[M_n(f)_p<]} is to use
Lemma \ref{Ineq:[|f - S_n(f)|_p]} and study the function defined as
\[
F_m:=Q_m(f)-S_m(f).
\]
 Clearly, the triangular inequality tells us
that
\begin{equation} \nonumber
\|f - Q_m(f)\|_p \
\le \ \|f - S_m(f)\|_p \ + \ \|F_m\|_p.
\end{equation}
Therefore, the proof of Theorem \ref{Theorem:[M_n(f)_p<]} hinges on
obtaining an estimate for $L_p(\TT)$-norm of the function $F_m$. To
this end, we recall some useful facts about trigonometric
polynomials and Fourier series.

We denote by $\Tt_m$ the space of univariate trigonometric
polynomials of degree at most $m$. That is, we have that
$\Tt_m:=\mbox{span}\{\chi_l:|l|\in Z[m+1]\}.$ We  require a readily
verified quadrature formula which says, for any $f\in\Tt_s$, that
\[
\hat{f}(0)=\frac{1}{s}\sum_{l\in Z[s]}f\left(\frac{2\pi
l}{s}\right).
\]
Using these facts leads to a formula from \cite{DD91}
which we state in the next lemma.

\begin{lemma} \label{Lemma:[(f*g)]}
If $m,n,s\in\NN$, such that $m + n < s$ then for any $f_1 \in \Tt_m$
and $f_2 \in \Tt_n$ there holds the following identity
\begin{equation} \nonumber
f_1*f_2 \ = \ s^{-1}\sum_{l\in Z[s]} f_1\left(\frac{2\pi
l}{s}\right) f_2\left(\cdot - \frac{2\pi l}{s}\right).
\end{equation}
\end{lemma}

Lemma \ref{Lemma:[(f*g)]} is especially useful to us as it gives a
convenient representation for the function $F_m$. In fact, it
readily follows, for $f=\kappa_r*g$, that
\begin{equation}\label{F_m(f)}
F_m=\frac{1}{2m+1}\sum_{l\in Z[m+1]}S_m(g)\left(\frac{2\pi
l}{2m+1}\right)\theta_m\left(\cdot-\frac{2\pi l}{2m+1}\right),
\end{equation}
where the function $\theta_m$ is defined as
$\theta_m:=\kappa_r-S_m(\kappa_r)$. The proof of formula
(\ref{F_m(f)}) may be based on the equation
\begin{equation}\label{S_m(k_r*g)}
S_m(\kappa_r*g)=\frac{1}{2m+1}\sum_{l\in Z[2m+1]}S_m(g)\left(\frac{2\pi
l}{2m+1}\right)(S_m\kappa_r)\left(\cdot-\frac{2\pi l}{2m+1}\right).
\end{equation}
For the confirmation of (\ref{S_m(k_r*g)}) we use the fact that
$S_m$ is a projection onto $\Tt_m$, so that
\[
S_m(\kappa_r*g)=S_m(\kappa_r)*S_m(g).
\]
Now, we use Lemma \ref{Lemma:[(f*g)]} with $f_1=S_m(g),
f_2=S_m(\kappa_r)$ and $s=2m+1$ to confirm both (\ref{F_m(f)}) and
(\ref{S_m(k_r*g)}).

The next step in our analysis makes use of equation (\ref{F_m(f)})
to get the desired upper bound for $\|F_m\|_p$. For this purpose, we
need to appeal to two well-known facts attributed to Marcinkiewicz,
see, for example, \cite{Zy59}. To describe these results, we
introduce the following notation. For any subset ${\bA}$ of $\bZ$
and a vector ${\bold a}:=(a_l:l\in\bA)$ and $1\leq p\leq \infty$ we
define the $l^p(\bA)$-norm of ${\bold a}$ by
\[
\|{\bold a}\|_{p,\bA} \ := \left\{\begin{array}{ll}
&\left(\sum_{l\in \bA}\, |a_l|^p\right)^{1/p},\quad 1\leq
p<\infty,\vspace{4mm}\\
 &\sup\{|a_l|:l\in\bA\}, \quad p=\infty.
\end{array}\right.
\]
Also, we introduce the mapping $W_m:{\cal T}_m\rightarrow \bR^{2m}$ defined at
$f\in{\cal T}_m$ as
$$
W_m(f)=\left(f\left(\frac{2\pi
l}{2m+1}\right):l\in Z[2m+1]\right).
$$
\begin{lemma} \label{Lemma:[|f|_p><]}
If $1 < p < \infty$, then there exist positive constants $c$ and $c'$
such that for any $m\in\bN$ and $f \in \Tt_m$ there hold the
inequalities
\[
c \, \|f\|_p \ \le (2m+1)^{-1/p}\|W_m(f)\|_{p, Z[2m + 1]} \ \le \ c' \, \|f\|_p.
\]
\end{lemma}
\begin{remark}
Lemma \ref{Lemma:[|f|_p><]} appears in \cite{Zy59} page 28, Volume
II as Theorem 7.5. We also remark in the case that $p=2$ the
constants appearing in Lemma \ref{Lemma:[|f|_p><]} are both one.
Indeed, we have for any $f\in\Tt_m$ the equation
\begin{equation} \label{W_mf}
(2m+1)^{-1/2}\|W_m(f)\|_{2, Z[2m+1]}=\|f\|_2.
\end{equation}
\end{remark}

\begin{lemma} \label{Lemma:[MMTh]}
If $1 < p < \infty$ and there is a positive constant $c$ such that
for any vector $\bold a = (a_j:j \in \ZZ)$ which satisfies for
some positive constant $A$ and any $s\in \mathbb{Z}$, the condition
\[
\ \ \sum_{j = \pm 2^s}^{\pm 2^{s+1}- 1}|a_j - a_{j-1}| \ \le A,
\]
and  also $ \|\bold a \|_{\infty,\bZ}\leq A $, then for any
functions $f \in L_p(\TT)$, the function
\[
M_{\bold a}(f) \ := \ \sum_{j \in \ZZ} a_j \hat{f}(j) \chi_j
\]
belongs to $L_p(\TT)$ and, moreover, we have that
\[
 \|M_{\bold a}(f)\|_p
\ \le \ c\,A \|f\|_p.
\]
\end{lemma}
\begin{remark}
Lemma \ref{Lemma:[MMTh]} appears in \cite{Zy59} page 232, Volume II
as Theorem 4.14 and is sometimes referred as the Marcinkiewicz
multiplier theorem.
\end{remark}

We are now ready to prove Theorem \ref{Theorem:[M_n(f)_p<]}.

\begin{proof}
For each $j\in\bZ$ we define
\begin{equation}\label{b-j} 
b_j :=
\frac{1}{2m + 1}\sum_{l \in Z[2m+1]} S_m(g)\left(\frac{2\pi
l}{2m+1}\right) e^{-\frac{2\pi ilj}{2m+1}}
\end{equation}
and observe from equation (\ref{F_m(f)}) that
\begin{equation} \label{Eq:[F_m]}
F_m \ = \sum_{j \in \bar{Z}[m]}b_j|j|^{-r}\chi_j,
\end{equation}
where $\bar{Z}[m]:= \{j \in \bZ: |j| > m\}$.
Moreover, according to equation (\ref{b-j}), we have for every
$j\in\bZ$ that
\begin{equation}\label{b_j-periodic}
b_{j+2m+1}=b_j.
\end{equation}
Notice that 
$\left(S_m(g)\left(\frac{2\pi l}{2m+1}\right): l \in Z[2m+1]\right)$ is the discrete Fourier transform of 
$\left(b_j: j \in Z[2m+1]\right)$
and therefore, we get for all $l\in Z[2m+1]$ that
\begin{equation}\nonumber
S_m(g)\left(\frac{2\pi l}{2m+1}\right)=\sum_{j\in Z[2m+1]}b_je^{\frac{2\pi ilj}{2m+1}}.
\end{equation}
On the other hand, by definition we have
\begin{equation}\nonumber
S_m(g)\left(\frac{2\pi l}{2m+1}\right)=\sum_{|j|\in Z[m+1]}\hat{g}(j) e^{\frac{2\pi ilj}{2m+1}}.
\end{equation}
Hence,
\begin{equation}\label{eq:[b_j]}
b_j
\ = \ 
\begin{cases}
\hat{g}(j), \ & 0 \le j \le m, \\
 \hat{g}(j- 2m -1), \ & m+1 \le j \le 2m.
\end{cases}
\end{equation}

We decompose the set 
$\bar{Z}[m]$ as a disjoint union of finite sets each
containing $2m+1$ integers. Specifically, for each $j \in \bZ$ we
define the set
\[
I_{m,j}:= \{ l:l\in \bN, j(2m + 1) - m \le l \le
j(2m + 1) + m \}
\]
 and observe that $\bar{Z}[m]$ is a disjoint
union of these sets. Therefore, using equations \eqref{Eq:[F_m]} and (\ref{b_j-periodic})
we can compute
\begin{equation} \nonumber
F_m \ = \ \sum_{j\in\bar{Z}[0]} \, \sum_{l \in I_{m,j}}
|l|^{-r}\,b_l \, \chi_l =\sum_{j\in\mathbb{N}} G_{m,j}\chi_{j(2m+1)-m}
\end{equation}
where
\[
G_{m,j}\ := \ \sum_{l \in Z[2m+1]}|l +  j(2m + 1) - m|^{-r}\,b_l \chi_l.
\]
Hence, by the triangle inequality we conclude that
\begin{equation} \label{Ineq:[|F_m^+|_p]}
\|F_m\|_p \ \le \ \sum_{j\in\bar{Z}[0]} \|G_{m,j}\|_p.
\end{equation}
By using \eqref{b_j-periodic} and \eqref{eq:[b_j]} we split the function $G_{m,j}$ into two functions as follows
\begin{equation} \label{split:[G_{m,j}]}
G_{m,j}
\ = \
G_{m,j}^+ + \chi_{2m+1}G_{m,j}^-,
\end{equation}
where
\[ G_{m,j}^+ := \
\sum_{l=0}^m |l +  j(2m + 1) - m|^{-r}\, \hat{g}(l) \chi_l,
\quad
G_{m,j}^- := \
\sum_{l=-m}^{-1} |l +  (j+1)(2m + 1) - m|^{-r}\, \hat{g}(l) \chi_l.
\]

Now, we shall use Lemma
\ref{Lemma:[MMTh]} to estimate $\|G_{m,j}^+\|_p$ and $\|G_{m,j}^-\|_p$. For this purpose, we
define for each $j\in\bN$ the components of a vector ${\bold a}=(a_l:l\in\bZ)$ as
\begin{equation*}
 a_l
 \ := \
\begin{cases}
|l +  j(2m + 1) - m|^{-r}, & l \in Z[m+1], \\
0, & \  \text{otherwise},
 \end{cases}
\end{equation*}
we may conclude that  Lemma \ref{Lemma:[MMTh]} is applicable when a value of $A$ is specified. For simplicity let us consider the case $j > 0$, the other case can be treated in a similar way.
For a fixed value of $j$ and $m$, we observe that the components of
the vector ${\bold a}$ are decreasing with regard to $|l|$ and moreover, it is readily
seen that $a_0\leq |jm|^{-r}$. Therefore, we may choose
$A=|jm|^{-r}$ and apply Lemma \ref{Lemma:[MMTh]} to conclude that 
$\|G_{m,j}^+\|_p \ \leq \ \rho|jm|^{-r} \|f\|_{K^r_p(\bT)}$
where $\rho$ is a constant which is independent of $j$ and $m$. The same inequality can be obtained for $\|G_{m,j}^-\|_p$. Consequently, by \eqref{split:[G_{m,j}]} and the triangle inequality we have 
\begin{equation}\nonumber
\|G_{m,j}\|_p \ \leq \ 2\rho|jm|^{-r} \|f\|_{K^r_p(\bT)}.
\end{equation}
We combine this inequality with 
inequalities (\ref{Ineq:[|F_m^+|_p]}) and $r>1$ to
conclude, that there is positive constant $c$, independent of
$m$, such that 
\[
 \|F_m\|_p\leq c\, m^{-r}\|f\|_{K^r_p(\bT)}. 
\]

We now turn our attention to the proof of inequality
(\ref{Q_m(g)_p}). Since $\kappa_r$ is continuous on $\bT$, the proof
of (\ref{Q_m(g)_p}) is transparent. Indeed, we successively use the
H$\ddot{o}$lder inequality, the upper bound in Lemma
\ref{Lemma:[|f|_p><]} applied to the function $S_m(g)$ and the
inequality (\ref{S_m(g)}) to obtain the desired result.
\end{proof}

\begin{remark}
The restrictions $1 < p < \infty$ and $r > 1$ in  Theorem \ref{Theorem:[M_n(f)_p<]} are necessary for applying the Marcinkiewicz multiplier theorem (Lemma \ref{Lemma:[MMTh]}) and processing the upper bound of $\|F_m\|_p$. It is interesting to consider this theorem for the case $0 < p \le \infty$ and $r > 0$. However, this would go beyond the scope of this paper.
\end{remark}

We end this section by providing an improvement of Theorem
\ref{Theorem:[M_n(f)_p<]} when $p=2$.
\begin{theorem} \label{Theorem:[M_n(f)_2<]}
If $\ r > 1/2$, then there is a positive constant $c$ such that for
all $m \in \mathbb{N}$ and $f \in K^r_2(\bT)$, we have that
\begin{equation} \nonumber
\|f - Q_m(f)\|_2 \ \le \ c \,  m^{-r}\|f\|_{K^r_2(\bT)}.
\end{equation}
\end{theorem}

\begin{proof}
This proof parallels that given for Theorem
\ref{Theorem:[M_n(f)_p<]} but, in fact, is simpler. From the
definition of the function $F_m$ we conclude that
\begin{equation} \nonumber
\|F_m\|_2^2 \ =\sum_{j \in \bar{Z}[m]} |j|^{-2r}|b_j|^2 = \
\sum_{j\in\mathbb{N}} \, \sum_{k \in I_{m,j}} |k|^{-2r}\, |b_k|^2 .
\end{equation}
We now use equation (\ref{b_j-periodic}) to obtain that
\begin{equation} \nonumber
\begin{aligned}
\|F_m\|_2^2 \  &= \ \sum_{j\in  \bar{Z}[m]} \,
\sum_{k \in Z[2m+1]} |k + j(2m + 1) - m|^{-2r}\, |b_k|^2 \\
& \leq m^{-2r}\sum_{j\in \bar{Z}[0]} |j|^{-2r} \sum_{k \in Z[2m+1]} |b_k|^2
\ \ll \ m^{-2r} \sum_{k \in Z[2m+1]} |b_k|^2.
\end{aligned}
\end{equation}
Hence, appealing to Parseval's identity for discrete Fourier transforms applied to the pair 
$\left(b_k: k \in Z[2m+1]\right)$ and $\left(S_m(g)\left(\frac{2\pi l}{2m+1}\right): l \in Z[2m+1]\right)$, 
and \eqref{W_mf} we finally
get that
$$
\|F_m\|_2^2 \ll m^{-2r}\|g\|_2^2 = m^{-2r}\|f\|^2_{K_p^r(\bT)}
$$
which completes the proof.
\end{proof}

\section{Multivariate Approximation} \label{Multivariate approximations}
\setcounter{equation}{0}

Our goal in this section to make use of our univariate operators  and create multivariate operators from them  which  economize on the number of translates of $K_{r,d}$ used to approximate  while maintaining as high an order of approximation. To this end, we apply, in the present  context,  the techniques of Boolean sum approximation. These ideas go back to Gordon \cite{G77} for surface design and also  Delvos and  Posdorf \cite{DP77} in the 1970's.  Later, they appeared, for example, in the papers \cite{W78, MW81, CCM89} and because of their importance continue to attract interest and applications. We also employ hyperbolic cross and sparse grid techniques which
date back to Babenko \cite{Ba60} and Smolyak \cite{S63} to construct methods of multivariate approximation. These techniques then were widely used in numerous papers of Soviet mathematicians (see surveys in \cite{DD86,DD91b, Te93} and bibliography there) and have been developed in \cite{DD01, DD11, DU13, SU07, SU09, SU11} for hyperbolic cross approximations and sparse grid sampling recoveries. Our construction of approximation methods is a modification of those given in \cite{DD91b, DD11} (cf. \cite{SU07, SU09, SU11}). For completeness let us give its detailed description.  

For our presentation we find it convenient to
express the linear operator $Q_m$ defined in equation (\ref{Q_mf})
in an alternate form. Our preference here is to introduce a kernel
$H_m$ on $\bT^2$ defined for $x,t\in\bT$ as
\begin{equation}\nonumber
H_m(x,t)=\frac{1}{2m+1}\sum_{l\in Z[2m+1]}\kappa_r\left(x-\frac{2\pi
l}{2m+1}\right)D_m\left(\frac{2\pi l}{2m+1}-t\right)\end{equation}
and then observe when $f=\kappa_r*g$ for $g\in L_p(\bT)$ that
\begin{equation}\nonumber
Q_m(f)(x)=\int_{\bT}H_m(x,t)g(t)dt.
\end{equation}

For each lattice vector ${\bold
m}=(m_j:j\in N[d])\in\bN^d$ we form the operator
\begin{equation*}
Q_{\bold m}
:=\prod_{l\in N[d]}Q_{m_l},
\end{equation*}
where the univariate operator
$Q_{m_l}$ is applied to the univariate function $f$ by considering $f$ as a 
function of  variable $x_l$ with the other variables held fixed. This definition is adequate since the operators 
$Q_{m_l}$ and $Q_{m_{l'}}$ commute for different $l$ and $l'$. Below we will shortly define other operators in this fashion without explanation. 

We introduce a
kernel $H_{\bold m}$ on $\bT^d\times\bT^d$ defined at ${\bold
x}=(x_j:j\in N[d])$, ${\bold t}=(t_j:j\in N[d])\in\bT^d$ as
\begin{equation}\nonumber
H_{\bold m}({\bold x},{\bold t}):=\prod_{j\in N[d]}H_{m_j}(x_j,{t_j})
\end{equation}
and conclude for $f\in K_p^r(\bT^d)$ represented as $f=\kappa_{r,d}*g$
where $g\in L_p(\bT^d)$ and ${\bold x}\in\bT^d$ we get that
\[
Q_{\bold m}(f)({\bold x})=\int_{\bT^d}H_{\bold m}({\bold x},{\bold t}){\bold g}({\bold t})d{\bold t}.
\]

To assess the error in approximating the function $f$ by the function
$Q_{\bold m}f$ we need a convenient representation for $f-Q_{\bold m}f$. Specifically, for each nonempty subset $\bV\subseteq N[d]$ and lattice vector ${\bold m}=(m_l:l\in N[d])\in\bN^d$, we let $|\bV|$
be the cardinality of $\bV$ and define linear operators
\begin{equation}\nonumber
Q_{{\bold m},\bV}:=\prod_{l\in\bV}(I-Q_{m_l}).
\end{equation}
Consequently, it follows that
\begin{equation}\label{I-Q_m}
I-Q_{\bold m}=\sum_{\bV\subseteq N[d]}(-1)^{|\bV|-1}Q_{{\bold
m},\bV},
\end{equation}
where the sum is over all nonempty subsets of $N[d]$. To make use
of this formula, we need the following lemma.

\begin{lemma}\label{Lemma:Q_jm_j}
If $1<p<\infty$ but $p\neq2$ and $r>1$ or $r>1/2$ when $p=2$, $d$ is
a positive integer and $\bV$ is a nonempty subset of $N[d]$, then
there exists a positive constant $c$ such that for any ${\bold
m}=(m_j:j\in\bN^d)$ and $f\in K_p^r(\bT^d)$ we have that
\begin{equation}\nonumber
\|Q_{{\bold m},\bV}(f)\|_p
\leq
\frac{c}{(\prod_{l\in\bV}m_l)^r}\|f\|_{K_p^r(\bT^d)}.
\end{equation}
\end{lemma}

\begin{proof}
First, we return to the univariate case and introduce a kernel
$W_{r,m}$ on $\bT^2$ defined at $x,t\in\bT$ as
\[
W_{r,m}(x,t):=\kappa_r(x-t)-H_m(x,t).
\]
Consequently, we obtain for
$f=\kappa_r*g$ that
\[
f(x)-Q_m(f)(x)=\int_{\bT}W_{r,m}(x,t)g(t)dt.
\]
Therefore, by Theorems \ref{Theorem:[M_n(f)_p<]} and
\ref{Theorem:[M_n(f)_2<]} there is a positive constant $c$ such that
the integral operator ${\cal W}_{r,m}:L_p(\bT)\rightarrow L_p(\bT)$
defined at $g\in L_p$ and $x\in\bT$ to be
\[
{\cal W}_{r,m}(g)(x)=\int_{\bT}W_{r,m}(x,t)g(t)dt
\]
 has an operator norm
satisfying the inequality
\begin{equation}\label{cal_W}
\|{\cal W}_{r,m}\|_p:=\sup\{\|{\cal
W}_{r,m}(g)\|_{L_p}:\|g\|_{L_p}\leq 1\}\leq \frac{c}{m^r}.
\end{equation}
Our goal is to extend this circumstance to the multivariate operator
$Q_{{\bold m},\bV}$. We begin with the case that $|\bV|=1$. For
simplicity of presentation, and without loss of generality, we
assume that $\bV=\{1\}$. Also, we write vectors in $\bR^d$ in
concatenated form. Thus, we have ${\bold x}=(x_1,{\bold y}), {\bold
y}\in\bR^{d-1}$ and also ${\bold t}=(t_1,{\bold v}),{\bold
v}\in\bT^{d-1}$. Now, whenever $f=\kappa_r*g$ we may write it in the
form $$f(x,{\bold y})=\int_{\bT}\kappa_r(x_1-t_1)w({\bold
x}_1(t_1))dt_1$$ where $$ w(x_1,{\bold
y}):=\int_{\bT^{d-1}}\kappa_{r,d-1}({\bold y}-{\bold v})g(x_1,{\bold
v})d{\bold v}.$$ By Theorems \ref{Theorem:[M_n(f)_p<]} and
\ref{Theorem:[M_n(f)_2<]}  we are assured that there is
a positive constants $c_1$ such that for all ${\bold y}\in\bT^{d-1}$
we have that
\begin{equation}\label{f-Q_1,m_1}
\int_{\bT}|f(x_1,{\bold
y})-Q_{m_1}(f)(x_1,{\bold y})|^pdx_1\leq
\frac{c_1}{m_1^{rp}}\int_{\bT}|w(x_1,{\bold y})|^pdx_1.
\end{equation}
We must bound the integral appearing on the right hand side of the
above inequality. However, for $r>1$ we see that $\kappa_{r,d-1}\in
C(\bT^{d-1})$ while for $r>1/2$ we can only see that
$\kappa_{r,d-1}\in L_2(\bT^{d-1})$. Hence, in either case,
H$\ddot{o}$lder inequality ensures there is a positive constant
$c_2$ such that for all ${\bold x}, {\bold t}\in\bT^d$ we have that
\begin{equation}\label{w_(x_1,y)}
|w(x_1,{\bold y})|^p\leq c_2\int_{\bT^{d-1}}|g(x_1,{\bold v})|^pd{\bold v}.
\end{equation}
We now integrate both sides of (\ref{f-Q_1,m_1}) over
${\bold y}\in\bT^{d-1}$ and use inequality (\ref{w_(x_1,y)}) to
prove the result for $|\bV|=1$.

The remainder of the proof proceeds in a similar way. We illustrate
the steps in the proof when $\bV= N[s]$ and $s$ is some positive
integer in $N[d]$. This is essentially the general case by
relabeling the elements of $\bV$ when it has cardinality $s$. As
above, we write vectors in concatenate form ${\bold x}=({\bold
y}^1,{\bold y}^2)$ and ${\bold t}=({\bold v}^1,{\bold v}^2)$ where
${\bold y}^1, {\bold v}^1\in\bR^s$ and ${\bold v}^2, {\bold
y}^2\in\bR^{d-s}$. We require a convenient representation for the
function appearing in the left hand side of the inequality
(\ref{Lemma:Q_jm_j}). This is accomplished for functions $f$ having
the integral representation at ${\bold y}^1,{\bold y}^2\in\bT^s$,
given as
\[
f({\bold y}^1,{\bold y}^2)=\int_{\bT^s}\kappa_{r,s}({\bold
y}^1-{\bold v}^1)g({\bold v}^1,{\bold y}^2)d{\bold v}^1,
\]
where $g\in L_p(\bT^d)$. In that case, it is easily established for
${\bold y}^1=(y_l^1:l\in N[s])$ and ${\bold v}^1=(v_l^1:l\in N[s])$
that
$$
\prod_{l\in N[s]}(I-Q_{m_l})(f)({\bold y}^1,{\bold
y}^2)=\int_{\bT^s}\prod_{l\in N[s]}W_{r,m_l}(y_l^1, v_l^1)g({\bold
v}^1,{\bold y}^2)d{\bold v}^1.
$$
This is the essential formula that yields the proof of the lemma. We
merely use the operator inequality (\ref{cal_W}) and the method
followed above for the case $s=1$ to complete the proof.
\end{proof}

We draw two conclusions from this lemma. The first concerns an
estimate for the $L_p(\bT^d)$-norm of the function $f-Q_{\bold
m}(f)$. The proof of the next lemma follows directly from Lemma
\ref{Lemma:I-Q_jm_j} and equation (\ref{I-Q_m}).
\begin{lemma}\label{Lemma:I-Q_jm_j}
If $1<p<\infty$ but $p\neq2$ and $r>1$ or $p=2$ and $r>1/2$, then
there exists a positive constant $c$ such that for any $f\in
K_p^r(\bT^d)$, ${\bold m}=(m_j:j\in\bN^d)$ we have that
\[
\|f-Q_{\bold m}(f)\|_{L^p(\bT^d)}\leq \frac{c}{(\min\{m_j:j\in N[d]\})^r}\|f\|_{K_p^r(\bT^d)}.
\]
\end{lemma}

We now give another use of this Lemma \ref{Lemma:I-Q_jm_j} by
introducing a scale of linear operators on $K^r_p(\bT^d)$. First, we
choose a $k\in\bZ_+$ and define on $K^r_p(\bT)$ the linear operator
$T_k=I-Q_{2^k}$. Also, we set $T_{-1}=I$. Next, we choose a lattice
vector ${\bold k}=(k_j:j\in N[d])\in\bZ^d$, set $|{\bold
k}|=\sum_{j\in N[d]}k_j$ and on $K^r_p(\bT^d)$ we define the linear
operator
\[
 T_{\bold k}=\prod_{l\in N[d]}T_{k_l}.
\] 

The next lemma is an immediate consequence of Lemma
\ref{Lemma:Q_jm_j}.

\begin{lemma}  \label{lemma[T_k(f)<]}
If $1 < p < \infty, \ p \ne 2, \ r > 1$ or $p = 2, \ r > 1/2$, then
there is a positive constant $c$ such that for any $f \in \kr$ and
${\bold k} \in \ZZdp$, there holds the inequality
\begin{equation} \nonumber
\| T_{\bold k}(f)\|_p \ \le \ c \, 2^{-r|\bold k|} \|f\|_{\kr(\bT^d)}.
\end{equation}
\end{lemma}

In the next result, we provide a decomposition of the space
$K^r_p(\bT^d)$. To this end, for $k \in \bN$ we define on
$K^r_p(\bT)$ the linear operator $ R_k := \ Q_{2^k} -  Q_{2^{k-1}}$
and also set $R_0=Q_1$. Note that for $k\in\bZ_+$ we also have that
$R_k=T_{k-1}-T_k$. Moreover, it readily follows that
\[
Q_{2^k}=\sum_{l\in Z[k+1]}R_l.
\]
 Let us now extend this setup to
the multivariate case. For this purpose, if ${\bold
l}=(l_j:j\in N[d])$ and ${\bold k}=(k_j:j\in N[d])$ are lattice
vectors in $\bZ^d_+$, we say that ${\bold l}\prec{\bold k}$ provided
for each $j\in N[d]$ we have that $l_j\leq k_j$. Now, as above, for
any lattice vector ${\bold k}=(k_j:j\in N[d])$, we define on
$K^r_p(\bT^d)$ the linear operator 
\[ R_{\bold k}=\prod_{l\in N[d]}R_{k_l}
\]
 and observe that
\begin{equation}\label{Q_2^k}
Q_{2^{\bold k}}=\sum_{{\bold l}\prec{\bold k}}R_{\bold l}.
\end{equation}

We now are ready to describe our decomposition of
this space $K^r_p(\bT^d)$.

\begin{theorem}  \label{theorem[decomposition]}
If $1 < p < \infty, \ p \ne 2, \ r > 1$ or $p = 2, \ r > 1/2$, then
there exists a positive constant $c$ such that for every $f \in
K^r_p(\bT^d)$ and ${\bold k} \in \ZZdp$, we have that
\begin{equation} \label{ineq[q_k(f)<]} 
\| R_{\bold k}(f)\|_p \ \le \ c \, 2^{-r|{\bold k}|} \|f\|_{K^r_p(\bT^d)}.
\end{equation}
Moreover, every $f \in K^r_p(\bT^d)$ can be represented as
\begin{equation} \nonumber
f \ = \ \sum_{{\bold k} \in \ZZdp} R_{\bold k}(f),
\end{equation}
where the series converges in $L_p(\bT^d)$-norm.
\end{theorem}

\begin{proof}
According to equation (\ref{Q_2^k}) we have for any $f\in K^r_p(\bT^d)$ that
\begin{equation}\label{Q_2^kf}Q_{2^{\bold k}}
(f)=\sum_{{\bold l}\prec{\bold k}}R_{\bold l}(f).
\end{equation} 
If each component of ${\bold k}$ goes to
$\infty$, then by Lemma \ref{Lemma:I-Q_jm_j} we see that the left
hand side of equation (\ref{Q_2^kf}) goes to $f$ in the
$L_p(\bT^d)$-norm. Therefore, to complete the proof, we need only
confirm the inequality (\ref{ineq[q_k(f)<]}). To this end, for each
nonempty subset $\bV\subseteq N[d]$ and lattice vector ${\bold
k}=(k_l:l\in N[d])\in\bN^d$ we define a new lattice vector ${\bold
k}_{\bV}\in\bN^d$ which has components given as
\[
({\bold k}_{\bV})_j:=\left\{ 
\begin{array}{ll}
& k_j,\quad \quad \quad j\in\bV, \\ 
& k_j-1,\quad j\notin\bV.
\end{array}\right.
\]
Since $R_{\bold k}=\prod_{l\in N[d]}R_{k_l}$ and for $l\in N[d]$
we have that $R_{k_l}=T_{k_l-1}-T_{k_l}$, we obtain that
$$
R_{\bold k}=\sum_{\bV\subseteq N[d]}(-1)^{|\bV|}T_{{\bold k}_{\bV}},
$$
and so by Lemma \ref{lemma[T_k(f)<]} there exists a positive
constant $c$ such that for every $f \in K^r_p(\bT^d)$ and ${\bold k}
\in \ZZdp$ we have that
$$
\|R_{\bold k}(f)\|_p\leq \sum_{\bV\subseteq N[d]} \| T_{{\bold
k}_{\bV}}(f)\|_p \ \leq \ c\sum_{\bV\subseteq N[d]} 2^{-r|{\bold
k}_{\bV}|}\| f\|_{K^r_p(\bT^d)} \ \leq \ c2^{-r|{\bold k}|}\|
f\|_{K^r_p(\bT^d)}
$$
which completes the proof of this theorem.
\end{proof}

Our next observation concerns the linear operator $P_m$ defined on
$K^r_p(\bT^d)$ for each $m \in \ZZ_+$ as
\begin{equation} \label{P_m}
P_m := \ \sum_{|{\bold k}| \le m} R_{\bold k}.
\end{equation}

\begin{theorem}  \label{theorem[approximation]}
If $1 < p < \infty, \ p \ne 2, \ r > 1$ or $p = 2, \ r > 1/2$, then
there exists a positive constant $c$ such that for every $m \in
\ZZdp$ and $f \in K^r_p(\bT^d)$,
\begin{equation} \nonumber
\|f - P_m(f)\|_p \ \le \ c \, 2^{-rm} m^{d-1}\, \|f\|_{
K^r_p(\bT^d)}.
\end{equation}
\end{theorem}

\begin{proof}
From Theorem \ref{theorem[decomposition]} we deduce that there
exists a positive constant $c$ (perhaps different from the constant
appearing in the Theorem \ref{theorem[decomposition]}) such that for
every $f \in K^r_p(\bT^d)$ and ${\bold k} \in \ZZdp$, we have that
\begin{equation*}
\begin{aligned}
\|f - P_m(f)\|_p \ &= \ \|\sum_{|{\bold k}| > m} R_{\bold k}(f)\|_p \ \le \
\sum_{|{\bold k}| > m} \|R_{\bold k}(f)\|_p \\
\ &\leq \ c\sum_{|{\bold k}| > m} 2^{-r|{\bold
k}|}\|f\|_{K^r_p(\bT^d)} \ \leq \ c\|f\|_{K^r_p(\bT^d)}
\sum_{|{\bold k}| > m} 2^{-r|{\bold k}|}\\
\ &\leq \ c2^{-rm} m^{d-1}\, \|f\|_{K^r_p(\bT^d)}.
\end{aligned}
\end{equation*}
\end{proof}

\section{Convergence rate and optimality} \label{Convergence rate and optimality}
\setcounter{equation}{0}

We choose a positive integer $m\in\bN$, a lattice vector ${\bold
k}\in\bZ^d_+$ with $|{\bold k}|\leq m$ and another lattice vector
${\bold s}=(s_j:j\in N[d])\in\otimes_{j\in N[d]} Z[2^{k_j+1}+1]$
to define the vector $\bold{y}_{{\bold k},{\bold
s}}=\left(\frac{2\pi s_j}{2^{k_j+1}+1}:j\in N[d]\right)$. The
Smolyak grid on $\bT^d$ consists of all such vectors and is given as
$$
G^d(m) := \ \{{\bold y}_{{\bold k},{\bold s}}:|{\bold k}|\leq m,
{\bold s}\in\otimes_{j\in N[d]} Z[2^{k_j+1}+1]\}.
$$
A simple computation confirms, for $m\rightarrow\infty$ that
$$
|G^d(m)| = \sum_{|{\bold k}| \le m} \prod_{j\in N[d]} (2^{k_j
+ 1} + 1)\asymp2^dm^{d-1},$$ so, $G^d(m)$ is a sparse subset of a
full grid of cardinality $2^{dm}$. Moreover, by the definition of
the linear operator $P_m$ given in equation (\ref{P_m}) we see that
the range of $P_m$ is contained in the subspace
$$\span\{\kappa_{r,d}(\cdot-{\bold y}):{\bold y}\in G^d(m)\}.$$
Now, we are ready to prove the next theorem, thereby establishing
inequality (\ref{[mainresult1]}).

\begin{theorem} \label{theorem[M_n(U^r_p)_p]}
If $1 < p < \infty, \ p \ne 2, \ r > 1$ or $p = 2, \ r > 1/2$, then
\begin{equation} \nonumber
M_n(U^r_p, \kappa_{r,d})_p \ \ll \ n^{-r} (\log n)^{r(d-1)}.
\end{equation}
\end{theorem}

\begin{proof} If $n\in\bN$ and $m$ is the largest positive integer such that
$ |G^d(m)| \le n$, then $n\asymp2^mm^{d-1}$ and by Theorem
\ref{theorem[approximation]} we have that $$M_n(U^r_p(\bT^d),
\kappa_{r,d})\leq \sup\{\|f-P_m(f)\|_{L_p(\bT^d)}:f\in
U^r_p(\bT^d)\}\ll 2^{-rm}m^{d-1}\asymp n^{-r}(\log n)^{r(d-1)}.$$
\end{proof}

Next, we prepare for the proof of the lower bound of \eqref{[mainresult2]}  in
Theorem \ref{Maintheorem}. To this end, let $\mathbb{P}_q(\bR^l)$ be
the set of algebraic polynomials of total degree at most $q$ on
$\RR^l$, and $\EEm$ the subset of $\RR^m$ of all vectors ${\bold t} =
(t_j:j\in N[m])$ with components in absolute value one. That is, for
every $j\in N[m]$ we demand that $|t_j|=1$. We choose a polynomial
vector field ${\bold p}: \bR^l\rightarrow\bR^m$ such that each
component of the vector field ${\bold p}$ is in
$\mathbb{P}_q(\bR^l)$. Corresponding to this polynomial vector
field, we introduce the polynomial manifold in $\bR^m$ defined as
$\mathbb{M}_{m,l,q}:={\bold p}(\bR^l)$. That is, we have that
\begin{equation*}
\mathbb{M}_{m,l,q} := \ \{(p_j({\bold u}):j\in N[m]):p_j\in\mathbb{P}_q(\bR^l),j\in N[m],{\bold u}\in\bR^l\}.
\end{equation*}
We denote the euclidean norm of a vector ${\bold x}$ in $\bR^m$ as
$\|{\bold x}\|_2$. For a proof of the following lemma see
\cite{Ma05}.
\begin{lemma} \label{Lemma[forlowerbound]}
(V. Maiorov) If $m, l, q\in \bN$ satisfy the inequality
$l \log(\frac{4emq}{l})\le \frac{m}{4}$, then there is a vector ${\bold
t} \in \EEm$ and a positive constant $c$ such that
\begin{equation*}
\inf\{ \|{\bold t} - {\bold x}\|_2:{\bold x} \in \mathbb{M}_{m,l,q}
\}\ \ge \ c\, m^{1/2}.
\end{equation*}
\end{lemma}

\begin{remark}
If we denote the euclidean distance of ${\bold t}\in\bR^m$ to the
manifold $\mathbb{M}_{m,l,q}$ by $\operatorname{dist}_2({\bold
t},\mathbb{M}_{m,l,q})$, then the lemma of V. Maiorov above says that
\[
\sup\{\operatorname{dist}_2({\bold y},\mathbb{M}_{m,l,q}):{\bold y
}\in\bE^m\}\geq cm^{-\frac{1}{2}}.
\]
\end{remark}
\begin{theorem} \label{theorem[M_n(U^r_2)_2]}
If  $r > 1/2$, then we have that
\begin{equation} \label{[M_n(U^r_2)_2]}
n^{-r} (\log n)^{r(d-2)} \ \ll \ M_n(U^r_2)_2 \ \ll \ n^{-r} (\log
n)^{r(d-1)}.
\end{equation}
\end{theorem}

\begin{proof}
The upper bound of \eqref{[M_n(U^r_2)_2]} was proved in Theorem
\ref{theorem[M_n(U^r_p)_p]}, and so we only need to prove the lower
bound by borrowing a technique used in the proof of \cite[Theorem
1.1]{Ma05}. For every positive number $a$ we define a subset
$\bH(a)$ of lattice vectors given by
\begin{equation*}
\bH(a) := \ \Big\{{\bold k}: {\bold k}=(k_j:j\in N[d])\in\bZ^d,
\prod_{j\in N[d]} |k_j| \le a \Big\}.
\end{equation*}
Recall that, for $a\rightarrow\infty$, we have that $|\bH(a)|\asymp
a(\log a)^{d-1}$, see, for example, \cite{DD84}. To apply Lemma
\ref{Lemma[forlowerbound]}, we choose for any $n\in\bN$, $q=\lfloor n(\log n)^{-d+2}\rfloor+1$, 
$m=5(2d+1)\lfloor n\log n\rfloor$ and
$l=(2d+1)n$. With these choices we observe that
\begin{equation}\label{H(q)}
|\bH(q)|\asymp m
\end{equation} 
and
\begin{equation}\label{q} 
q\asymp m(\log m)^{-d+1}
\end{equation}
 as $n\rightarrow\infty$.
Also, we readily confirm that
$$
\lim_{n\rightarrow\infty}\frac{l}{m}\log\left(\frac{4emq}{l}\right)=\frac{1}{5}
$$
and so the hypothesis of Lemma \ref{Lemma[forlowerbound]} is
satisfied for $n\rightarrow\infty$.

Now, there remains the task of specifying the polynomial manifold
$\mathbb{M}_{m,l,q}$. To this end, we introduce the positive
constant $\zeta := q^{-r} m^{-1/2}$ and let $\mathbb{Y}$ be the set
of trigonometric polynomials on $\TTd$, defined by
\begin{equation*}
\mathbb{Y} := \ \Big\{\zeta \sum_{{\bold k} \in \bH(q)} t_{\bold k}
\chi_{\bold k}: {\bold t}= (t_{\bold k}:{\bold k} \in
\bH(q))\in\bE^{|\bH(q)|}\Big\}.
\end{equation*}
If $f\in\mathbb{Y}$ is written in the form  $f=\zeta\sum_{{\bold k}
\in \bH(q)}t_{\bold k} \chi_{\bold k}$, then $f=\kappa_{r,d}*g$ for
some trigonometric polynomial $g$ such that
\[
\|g\|_{L_2(\bT^d)}^2\leq\zeta^2\sum_{{\bold k}\in\bH(q)}|\lambda_{\bold k}|^2,
\]
 where
$\lambda_{\bold k}$ was defined earlier before Definition
\ref{kappa_rd}. Since
\[
\zeta^2\sum_{{\bold k}\in\bH(q)}|\lambda_{\bold k}|^2\leq
\zeta^2q^{2r}|\bH(q)|=m^{-1}|\bH(q)|,
\]
we see from equation (\ref{H(q)}) that there is a positive constant
$c$ such that $\|g\|_{L_2(\bT^d)}\leq c$ for all $n\in\bN$. So, we
can either adjust functions in $\bY$ by dividing them by $c$ or we
can assume without loss of generality that $c=1$. We choose the
latter possibility so that $\bY\subseteq U_2^r(\bT^d)$.

We are now ready to obtain a lower bound for $M_n(U_2^r(\bT^d))_2$.
We choose any $\varphi\in L_2(\bT^d)$ and let $v$ be any function
formed as a linear combination of $n$ translates of the function
$\varphi$. Thus, for some real constants $c_j\in\bR$ and vectors
${\bold y}_j\in\bT^d, j\in N[n]$ we have that
\[
v=\sum_{j\in N[n]}c_j\varphi(\cdot-{\bold y}_j).
\]
 By the Bessel
inequality we readily conclude for
\[
f=\zeta\sum_{{\bold k}\in \bH(q)}t_{\bold k}\chi_{\bold k}\in\bY
\]
 that
\begin{equation}
\label{f-v}\|f-v\|^2_{L_2(\bT^d)}\geq \zeta^2\sum_{{\bold
k}\in\bH(q)}\left|t_{\bold k}-\frac{\hat{\varphi}({\bold
k})}{\zeta}\sum_{j\in N[n]}c_j e^{i({\bold y}_j,{\bold
k})}\right|^2.
\end{equation}

We now introduce a polynomial manifold so that we can use Lemma
\ref{Lemma[forlowerbound]} to get a lower bound for the expressions
on the left hand side of inequality (\ref{f-v}). To this end, we
define the vector ${\bold c}=(c_j:j\in N[n])\in\bR^n$ and for each
$j\in N[n]$, let ${\bold z}_j=(z_{j,l}:l\in N[d])$ be a vector in
$\bC^d$ and then concatenate these vectors to form the vector
${\bold z}=({\bold z}_j:j\in N[n])\in\bC^{nd}$. We employ the
standard multivariate notation
\[
{\bold z}_j^{\bold k}=\prod_{l\in N[d]}z_{j,l}^{k_l}
\]
and require vectors ${\bold w}=({\bold c},{\bold
z})\in\bR^n\times\bC^{nd}$ and ${\bold u}=({\bold c}, \Re{\bold z},
\Im{\bold z})\in\bR^l$ written in concatenate form. Now, we
introduce for each ${\bold k}\in\bH(q)$ the polynomial ${\bold q}_k$
defined at ${\bold w}$ as
$$
{\bold q}_k({\bold w}):=\frac{\hat{\varphi}({\bold
k})}{\zeta}\sum_{j\in\bH(q)}c_j{\bold z}^{\bold k}.
$$
We only need to consider the real part of ${\bold q}_k$, namely,
$p_{\bold k}=\Re{\bold q}_k$ since we have that
\[
\inf\left\{\sum_{{\bold k}\in\bH(q)}\left|t_{\bold
k}-\frac{\hat{\varphi}({\bold
k})}{\zeta}\sum_{j\in N[n]}c_je^{i({\bold y}_j,{\bold
k})}\right|^2:c_j\in\bR,{\bold
y}_j\in\bT^d\right\}\geq\inf\left\{\sum_{{\bold
k}\in\bH(q)}\left|t_{\bold k}-p_{\bold k}({\bold u})\right|^2:{\bold
u}\in\bR^l\right\}.
\]
Therefore, by Lemma \ref{Lemma[forlowerbound]} and \eqref{q} we conclude there is
a vector ${\bold t}^0= (t^0_{\bold k}:{\bold k} \in
\bH(q))\in\bE^{h_q}$ and the corresponding function
\[
f^0=\zeta\sum_{{\bold k} \in \bH(q)}t^0_{\bold k} \chi_{\bold
k}\in\bY
\]
 for which there is a positive constant $c$ such that for
every $v$ of the form $v=\sum_{j\in N[n]}c_j\varphi(\cdot-{\bold
y}_j)$ we have that
\[
\|f^0-v\|_{L^2(\bT^d)}\geq c\zeta m^{\frac{1}{2}}=q^{-r}\asymp
n^{-r}(\log n)^{r(d-2)}
\]
which proves the result.
\end{proof}

\bigskip
{\bf Acknowledgments}
\ Dinh D\~ung's research is  funded by Vietnam National Foundation for Science and Technology Development (NAFOSTED) under  Grant 102.01-2012.15. Charles A. Micchelli's research is partially supported by US
National Science Foundation Grant DMS-1115523. The authors would like to thank the referees for a critical reading
of the manuscript and for several valuable suggestions which helped to improve its presentation.

\end{document}